\documentclass[oneside]{amsart}
\usepackage{preamble}

\title{Asymptotic dimension and hyperfiniteness of generic Cantor actions}

\date{\today}

\begin{document}

\author{Sumun Iyer$^\dag$}
\address{$^\dag$Carnegie Mellon University; \textnormal{sumuni@andrew.cmu.edu}}

\author{Forte Shinko$^*$}
\address{$^*$University of California Berkeley; \textnormal{forteshinko@berkeley.edu}}

\blfootnote{This research was partially conducted at the Spring 2023 thematic program in Set Theoretic Methods in Algebra, Dynamics and Geometry at the Fields Institute for Research in Mathematical Sciences. The first author was supported by NSF GRFP grant DGE – 2139899 and the second author was supported by the Fields Institute.}

\begin{abstract}
    We show that for a countable discrete group
    which is locally of finite asymptotic dimension,
    the generic continuous action on Cantor space
    has hyperfinite orbit equivalence relation.
    In particular,
    this holds for free groups,
    answering a question of Frisch-Kechris-Shinko-Vidny\'anszky.
\end{abstract}
\maketitle

For this entire article,
fix a countable discrete group $\Gamma$.

\section{Introduction}
A \textbf{countable Borel equivalence relation (CBER)}
is an equivalence relation $E$ on a standard Borel space $X$
which is Borel as a subset of $X^2$,
and for which every equivalence class is countable
(see \cite{Kec24} for more background on CBERs).

The theory of CBERs seeks to classify these equivalence relations
based on their relative complexity.
More precisely,
there is a natural preorder on CBERs,
called the \textbf{Borel reducibility} preorder,
defined as follows:
if $E$ and $F$ are CBERs on $X$ and $Y$ respectively,
then $E \le_B F$ if there is a Borel map $f : X \to Y$
such that for all $x, x' \in X$,
we have
\[
    x \mathrel E x'
    \iff f(x) \mathrel F f(x').
\]
If $E \le_B F$,
then we think of $E$ as ``simpler'' than $F$.

The simplest CBERs are the so-called \textbf{smooth} CBERs,
which are those CBERs $E$ satisfying $E \le_B \Delta_{\R}$,
where $\Delta_{\R}$ is the equality relation on $\R$.
The canonical non-smooth CBER is $E_0$ on $2^\N$ defined as follows:
\[
    x \mathrel E_0 y \iff \exists k \forall n > k \; [x_n = y_n]
\]
A CBER $E$ is \textbf{hyperfinite} if $E \le_B E_0$.
Hyperfiniteness is the next level up from smoothness in the following sense:
by the Harrington-Kechris-Louveau theorem,
a CBER $E$ is non-smooth iff $E_0 \le_B E$
(see \cite[Theorem 6.5]{Kec24}).
Hyperfiniteness is a very active area of research,
in part due to the deep connection with amenability.
Given a Borel action $\Gamma \car X$ on a standard Borel space $X$,
denote by $E_\Gamma^X$ the \textbf{orbit equivalence relation} of $X$.
By the Connes-Feldman-Weiss theorem \cite{CFW81},
every orbit equivalence relation
of every amenable group is measure-hyperfinite,
where a CBER $E$ on $X$ is \textbf{measure-hyperfinite}
if for every Borel probability measure $\mu$ on $X$,
there is a Borel subset $Y \subseteq X$ with $\mu(Y) = 1$
such that $E\uhr Y$ is hyperfinite.
A long-standing open question of Weiss
asks whether we can remove the measure condition
(see \cite[Problem 17.8]{Kec24}):
\begin{weiss}
    Is every orbit equivalence relation of every amenable group hyperfinite?
    \label{weiss}
\end{weiss}
This problem is far from being resolved,
and in fact it is still open for solvable groups,
although it is known that the answer is positive for nilpotent and polycyclic groups
(see \cite[Corollary 7.5]{CJMST23}).
To answer Weiss's Question in the positive,
it would be enough to have a positive resolution to the following question
(see \cite[Problem 17.7]{Kec24}):
\begin{qn}\label{meashyp-is-hyp}
    Is every measure-hyperfinite CBER hyperfinite?
\end{qn}
It is possible for this question to have a strong negative answer.
For instance,
measure-hyperfiniteness is a $\PI^1_1$ property,
and it is possible that hyperfiniteness is $\SIGMA^1_2$-complete
(see \cite[6.1(C)]{DJK94}),
which in particular would imply that there are ``many''
measure-hyperfinite CBERs which are not hyperfinite.

Another possible approach to a strong negative answer
is to apply Baire category in a Polish space of CBERs,
which we make precise.
A \textbf{Cantor action} of $\Gamma$ is a group homomorphism $\Gamma \to \Homeo(2^\N)$,
which we view as a continuous action $\Gamma \car 2^\N$.
Viewing $\Homeo(2^\N)$ as a Polish group with the compact-open topology,
let $\Act(\Gamma)$ be the Polish subspace of
$\Homeo(2^\N)^\Gamma$ consisting of Cantor actions of $\Gamma$.
It was shown by Suzuki (see \cite[Corollary 2.4]{Suz17})
that if $\Gamma$ is an \textbf{exact} group,
meaning that its reduced C*-algebra is exact,
then the set
$\{\mbf a \in \Act(\Gamma) : \text{$\mbf a$ is measure-hyperfinite}\}$
is comeager,
where we say that an action is \textbf{hyperfinite}
or \textbf{measure-hyperfinite}
if its orbit equivalence relation is.
This raises the following natural question,
which appears as Problem 8.0.16 in \cite{FKSV23}
(note that there it is stated in terms of the space of subshifts,
but this is is equivalent by a result of Hochman,
see \cite[Theorem 4.4.12]{FKSV23}):
\begin{qn}\label{generic-hyperfinite}
    If $\Gamma$ is an exact group,
    is the set
    $\{\mbf a \in \Act(\Gamma) : \text{$\mbf a$ is hyperfinite}\}$
    comeager?
\end{qn}

A negative answer to \Cref{generic-hyperfinite}
would immediately give a negative answer to \Cref{meashyp-is-hyp}.
However,
we show that \Cref{generic-hyperfinite} has a positive answer
for a wide class of exact groups,
in particular for free groups,
for which the problem had been open:
\begin{thm}\label{main-thm}
    If $\Gamma$ is locally of finite asymptotic dimension,
    then the set
    $\{\mbf a \in \Act(\Gamma) : \text{$\mbf a$ is hyperfinite}\}$
    is comeager.
\end{thm}
Asymptotic dimension is a coarse invariant of discrete groups
taking values in $\N \cup \{\infty\}$.
Groups which are locally of finite asymptotic dimension
include free groups,
hyperbolic groups,
and mapping class groups of finite type surfaces
(see \cite[Part II]{BD08}
for background on asymptotic dimension of groups).
In particular,
this theorem exhibits examples of amenable groups,
such as $(\Z/2 \wr \Z)^2$,
for which it is now known that the generic Cantor action is hyperfinite,
but for which it is open whether all of its Cantor actions are hyperfinite.
There are still many amenable groups,
such as the solvable group $\Z \wr \Z$,
for which it not yet known that the generic Cantor action is hyperfinite.

\section{Background}
We denote by $\Homeo(2^\N)$ the homeomorphism group of the Cantor space $2^\N$,
viewed as a Polish group with the compact-open topology.

We describe an explicit basis for $\Homeo(2^\N)$.
We view every $\phi \in \Homeo(2^\N)$ as a directed graph
whose vertex set is $2^\N$,
and where there is a directed edge from $x$ to $y$ iff $\phi(x) = y$.
Ranging over all finite directed graphs $G$
and over all continuous maps $c : 2^\N \to V(G)$,
the sets
\[
    \{\phi \in \Homeo(2^\N)
    : \text{$c$ is a homomorphism of directed graphs from $\phi$ to $G$}\}
\]
form an open basis for the topology of $\Homeo(2^\N)$.

For the rest of this section,
fix a countable group $\Gamma$.

We write $S \Subset \Gamma$
to mean that $S$ is a finite subset of $\Gamma$
such that $1 \in S$ and $S^{-1} = S$.

Let $\Act(\Gamma)$ be the set of continuous actions of $\Gamma$ on $2^\N$.
We view $\Act(\Gamma)$ as the $\Homeo(2^\N)$-invariant Polish subspace of $\Homeo(2^\N)^\Gamma$
consisting of all group homomorphisms $\Gamma \to \Homeo(2^\N)$,
where the action $\Homeo(2^\N) \car \Homeo(2^\N)^\Gamma$ is by conjugation on each coordinate.

We describe an explicit basis for $\Act(\Gamma)$.
A \textbf{$\Gamma$-graph} is a pair $G = (V(G), E(G))$,
where $V(G)$ is a set,
and $E(G)$ is a subset of $\Gamma \times V(G) \times V(G)$.
We view every action $\Gamma \car X$ as an $\Gamma$-graph $G$
where $V(G) = X$
and $(\gamma, x, y) \in E(G)$ iff $\gamma \cdot x = y$.
For $\Gamma$-graphs $G$ and $G'$,
a function $f : V(G) \to V(G')$ is an \textbf{$\Gamma$-map} from $G$ to $G'$
if for every $(\gamma, v, w) \in E(G)$,
we have $(\gamma, f(v), f(w)) \in E(G')$.
A \textbf{finite $\Gamma$-graph} is a $\Gamma$-graph $G$
such that $V(G)$ is finite
and such that $E(G)$ is a cofinite subset of $\Gamma \times V(G) \times V(G)$.
Ranging over all finite $\Gamma$-graphs $G$
and over all continuous maps $c : 2^\N \to V(G)$,
the sets
\[
    \{\mbf a \in \Act(\Gamma) :
        \text{$c$ is a $\Gamma$-map from $\mbf a$ to $G$}\}
\]
form an open basis for the topology of $\Act(\Gamma)$.

\section{Locally checkable labelling problems}
For this section,
fix a countable group $\Gamma$.

We describe another basis for $\Act(\Gamma)$.

\begin{defn}
    An \textbf{LCL} on $\Gamma$
    (short for \textbf{Locally Checkable Labelling problem})
    is a set of functions
    each of whose domains is a finite subset of $\Gamma$.
\end{defn}
We think of an LCL as a set of ``allowed patterns'' for a coloring.
\begin{defn}
    Let $\Gamma \car X$ be an action,
    and let $\Pi$ be an LCL on $\Gamma$.
    A function $c$ with domain $X$ is a \textbf{$\Pi$-coloring}
    if there is some finite $\Pi_0 \subseteq \Pi$
    such that for all $x \in X$,
    there is some $P \in \Pi_0$
    such that for all $\gamma \in \dom(P)$,
    we have $c(\gamma x) = P(\gamma)$.
\end{defn}
So $c$ is a $\Pi$-coloring
iff there is some finite $\Pi_0 \subseteq \Pi$
for which $c$ is a $\Pi_0$-coloring.

\begin{prop}\label{prop_basis}
    Ranging over all LCLs $\Pi$
    and over all continuous maps $c$ from $2^\N$ to a discrete space,
    the sets
    \[
        \{\mbf a \in \Act(\Gamma) : \text{$c$ is a $\Pi$-coloring}\}
    \]
    form an open basis for the topology of $\Act(\Gamma)$.
\end{prop}
\begin{proof}
    First we show that each such set is open.
    Let $\mbf a \in \Act(\Gamma)$,
    let $\Pi$ be an LCL,
    and let $c$ be a continuous map from $2^\N$ to a discrete space,
    such that $c$ is a $\Pi$-coloring of $\mbf a$.
    Fix a total order on $\Pi$,
    and let $f : 2^\N \to \Pi$ be the function defined as follows:
    for $x \in 2^\N$,
    let $f(x)$ be the first element $P \in \Pi$
    such that for all $\gamma \in \dom(P)$,
    we have $c(\gamma x) = P(\gamma)$.
    This is continuous since $c$ is continuous.
    Endow $\Pi$ with a $\Gamma$-graph structure as follows:
    say that $(\gamma, P, Q)$ is an edge
    if $(\gamma P) \cup Q$ is a function.
    Then $f$ is a $\Gamma$-map from $\mbf a$ to $\Pi$,
    and for every $\mbf b \in \Act(\Gamma)$ for which $f$ is a $\Gamma$-map,
    we have that $c$ is a $\Pi$-coloring of $\mbf b$.

    To show that it is a basis,
    we will show that every set in the previous basis is of the new form.
    Fix a finite $\Gamma$-graph $G$ and a continuous map $f : 2^\N \to V(G)$.
    Fix $S \Subset \Gamma$ such that
    $(\Gamma \setminus S) \times V(G) \times V(G) \subseteq E(G)$.
    Consider the LCL $\Pi$ consisting of all functions $P : S \to V(G)$
    such that $(s, P(1), P(s)) \in E(G)$,
    Then for every $\mbf a \in \Act(\Gamma)$,
    we have that $f$ is a $\Gamma$-map from $\mbf a$ to $V(G)$
    iff $f$ is a $\Pi$-coloring.
\end{proof}

Notice that every set of the form 
\[
    \{\mbf a \in \Act(\Gamma) :
    \text{$\mbf a$ has a continuous $\Pi$-coloring}\}
\]
is an open set
since it is a union of the basic open sets considered in \Cref{prop_basis}. 

We show that nonempty such sets are dense.
\begin{prop}\label{lcl-generic}
    Let $\Pi$ be an LCL on $\Gamma$.
    Then the following are equivalent:
    \begin{enumerate}
        \item $\Gamma$ has a $\Pi$-coloring.
        \item $\{\mbf a \in \Act(\Gamma) :
        \text{$\mbf a$ has a continuous $\Pi$-coloring}\}$ is nonempty.
        \item $\{\mbf a \in \Act(\Gamma) :
        \text{$\mbf a$ has a continuous $\Pi$-coloring}\}$ is dense.
    \end{enumerate}
\end{prop}

We will need the following.
\begin{prop}\label{free-generic}
    The set $\{\mbf a \in \Act(\Gamma) : \text{$\mbf a$ is free}\}$
    is dense $G_\delta$.
\end{prop}
For a proof of \Cref{free-generic},
see \cite[Lemma 2.1]{Suz17} or \Cref{free-rmk}.
\begin{proof}[Proof of \Cref{lcl-generic}]
    \leavevmode
    \begin{itemize}
    \item (1 $\implies$ 2):
    
        It suffices to find some zero-dimensional compact $\Gamma$-space
        with a continuous $\Pi$-coloring,
        since its product with $2^\N$ yields a Cantor action with the same property.
        
        Fix a finite subset $\Pi_0 \subseteq \Pi$ such that $\Gamma$ has a $\Pi_0$-coloring,
        and let $K = \bigcup_{P \in \Pi_0} \im(P)$.
        View $K^\Gamma$ as a compact $\Gamma$-space
        equipped with the action $(\gamma \cdot x)_{\delta} = x_{\delta\gamma}$.
        Then the compact $\Gamma$-invariant subspace of $K^\Gamma$
        defined by
        \[
            X = \{x \in K^\Gamma : \text{$x$ is a $\Pi_0$-coloring}\}
        \]
        is nonempty,
        and it has a continuous $\Pi_0$-coloring given by $c(x) = x_1$.

    \item (2 $\implies$ 3):
    
         The action $\Homeo(2^\N) \car \Act(\Gamma)$ is generically ergodic,
         i.e. has a dense orbit
         (see \cite[Proposition 4.4.2]{FKSV23}),
         so since this set is non-empty and $\Homeo(2^\N)$-invariant,
         it is dense.
         
    \item (3 $\implies$ 1):
    
        By \Cref{free-generic},
        there is a free $\mbf a \in \Act(\Gamma)$
        with a continuous $\Pi$-coloring.
        By freeness,
        there is a $\Gamma$-equivariant map $\Gamma \to \mbf a$,
        and the composition of this with the $\Pi$-coloring of $\mbf a$
        is a $\Pi$-coloring of $\Gamma$.
    \end{itemize}
\end{proof}

\begin{rmk}\label{free-rmk}
    We can also prove \Cref{free-generic} using LCLs.
    One can show using a coloring result like \cite[Lemma 2.3]{Ber23}
    that a zero-dimensional Polish $\Gamma$-space is free
    iff for every $\gamma \in \Gamma$,
    it has a continuous coloring for the LCL
    consisting of injections $\{1, \gamma\} \hra \{0, 1, 2\}$.
    Then \Cref{free-generic} immediately follows from
    (1 $\implies$ 2 $\implies$ 3)
    of \Cref{lcl-generic},
    whose proof never used \Cref{free-generic}.
\end{rmk}

\section{Asymptotic dimension and hyperfiniteness}
For this section,
fix a countable group $\Gamma$.

\begin{defn}
    Let $n \in \N$.
    An \textbf{$n$-coloring} is a function
    whose image is a subset of
    $\{0, 1, 2, \ldots, n - 1\}$.
\end{defn}
\begin{defn}
    Let $\Gamma \car X$ be an action and let $S \Subset \Gamma$.
    A function $c$ with domain $X$ is \textbf{$S$-separated}
    if there is a uniform bound on the sizes of the components
    of the graph with vertex set $X$
    where $x$ and $x'$ are adjacent iff $x' \in Sx$ and $c(x) = c(x')$.
\end{defn}

\begin{defn}
    The \textbf{asymptotic dimension} of an action
    $\Gamma \car X$ of a group on a set,
    denoted $\asdim(\Gamma \car X)$,
    is defined as follows:
    \[
        \asdim(\Gamma \car X)
        = \sup_{S \Subset \Gamma}
        \min\{n \in \N :
        \text{$\Gamma \car X$ has an $S$-separated $n$-coloring}\} - 1
    \]
\end{defn}

We define the asymptotic dimension of a group.
\begin{defn}
    The \textbf{asymptotic dimension} of a group $\Gamma$,
    denoted $\asdim(\Gamma)$,
    is the asymptotic dimension of the left-multiplication action
    $\Gamma \car \Gamma$.
\end{defn}

\begin{defn}
    A group $\Gamma$ is \textbf{locally of finite asymptotic dimension}
    if all of its finitely generated subgroups have finite asymptotic dimension.
\end{defn}

Note that every free action of $\Gamma$ has asymptotic dimension $\asdim(\Gamma)$.
In particular,
if $\Delta$ is a subgroup of $\Gamma$,
then $\asdim(\Delta \car \Gamma) = \asdim(\Delta)$.

Asymptotic dimension can be encoded by LCLs.
\begin{defn}
    Let $S \Subset \Gamma$ and let $n \in \N$.
    The LCL $\Pi_{S, n}$ is the set of $n$-colorings $P$ with $\dom(P)$ such that
    \begin{enumerate}[label=(\roman*)]
        \item $\dom(P)$ is a finite subset of $\Gamma$;
        \item $1 \in \dom(P)$;
        \item for every $\gamma \in \dom(P)$ with $P(\gamma) = P(1)$,
            we have $S \gamma \subseteq \dom(P)$.
    \end{enumerate}
\end{defn}

\begin{prop}\label{asdim-lcl}
    Let $\Gamma \car X$ be an action,
    let $S \Subset \Gamma$,
    and let $n \in \N$.
    Then every $\Pi_{S, n}$-coloring of $X$ is an $S$-separated $n$-coloring.
    Moreover,
    if the action is free,
    then the converse also holds.
\end{prop}
\begin{proof}
    Fix a function $c$ with domain $X$.
    
    Let $G$ be the graph with vertex set $X$
    where $x$ and $x'$ are adjacent iff $x' \in Sx$ and $c(x) = c(x')$.
    
    Suppose $c$ is a $\Pi_{S, n}$-coloring.
    Then $c$ is a $\Pi_0$-coloring for some finite $\Pi_0 \subseteq \Pi_{S, n}$.
    Let $x \in X$.
    Since $c$ is a $\Pi_0$-coloring,
    there is some $P \in \Pi_0$ such that for every $\gamma \in \dom(P)$,
    we have $P(\gamma) = c(\gamma x)$.
    Then $[x]_G \subseteq \dom(P) x$.
    Hence every component of $G$ has size at most $\max_{P \in \Pi_0} |\dom(P)|$.

    Now suppose that the action is free,
    and suppose that $c$ is an $S$-separated $n$-coloring.
    Then there is some $k \in \N$
    such that for every $x \in X$,
    the $G$-component $[x]_G$ of $x$ satisfies $[x]_G \subseteq S^k x$.
    Let $\Pi_0 \subseteq \Pi_{S, n}$ consist of those $P$
    with $\dom(P) \subseteq S^{k+1}$.
    Now suppose $x \in X$.
    Consider the function $P$
    with domain $\{\gamma \in \Gamma : \gamma x \in S[x]_G\}$
    defined by $P(\gamma) = c(\gamma x)$.
    Then $P \in \Pi_{S, n}$,
    and we have $\dom(P)x \subseteq S[x]_G \subseteq S^{k+1} x$,
    so by freeness we have $\dom(P) \subseteq S^{k+1}$,
    and hence $P \in \Pi_0$.
    Thus $c$ is a $\Pi_{S, n}$-coloring.
\end{proof}

For Cantor actions,
we use a topological version of asymptotic dimension.
\begin{defn}
    The \textbf{continuous asymptotic dimension} of a continuous action
    $\Gamma \car X$ on a topological space,
    denoted $\asdim_c(\Gamma \car X)$,
    is defined as follows:
    \[
        \asdim_c(\Gamma \car X)
        = \sup_{S \Subset \Gamma}
        \min\{n \in \N :
        \text{$\Gamma \car X$ has a continuous $S$-separated $n$-coloring}\} - 1
    \]
\end{defn}

\begin{thm}\label{subgroup-asdim}
    Let $\Delta \le \Gamma$ be a subgroup.
    Then the set
    \[
        \{\mbf a \in \Act(\Gamma) :
        \text{$\mbf a$ is free and
            $\asdim_c(\mbf a\uhr \Delta) = \asdim(\Delta)$}\}
    \]
    is dense $G_\delta$,
    and hence comeager.
\end{thm}
\begin{proof}
    For free $\mbf a \in \Act(\Gamma)$,
    we have $\asdim_c(\mbf a \uhr \Delta)
    \ge \asdim(\mbf a \uhr \Delta)
    = \asdim(\Delta)$,
    so we need only consider the inequality
    $\asdim_c(\mbf a\uhr \Delta) \le \asdim(\Delta)$.
    
    Freeness is dense $G_\delta$ by \Cref{free-generic}.
    If $\asdim(\Delta) = \infty$,
    then the set in question is just the set of free actions,
    so we are done.

    So suppose $\asdim(\Delta) < \infty$.
    By \Cref{asdim-lcl},
    for $\mbf a \in \Act(\Gamma)$,
    we have that $\mbf a$ is free and satisfies
    $\asdim_c(\mbf a\uhr \Delta) \le \asdim(\Delta)$
    iff $\mbf a$ is free and has a continuous $\Pi_{S, \asdim(\Delta) + 1}$-coloring
    for every $S \Subset \Delta$,
    The latter condition is dense $G_\delta$ by \Cref{lcl-generic},
    so we are done.
\end{proof}
In particular,
\[
    \{\mbf a \in \Act(\Gamma) :
    \text{$\mbf a$ is free and $\asdim_c(\mbf a) = \asdim(\Gamma)$}\}
\]
is dense $G_\delta$ and hence comeager,
so if $\asdim(\Gamma)$ is finite,
then the generic element of $\Act(\Gamma)$ is hyperfinite
by \cite[Theorem 7.1]{CJMST23}.
We can sharpen this to obtain \Cref{main-thm} from the introduction:
\begin{thm}
    The set
    \[
        \{\mbf a \in \Act(\Gamma) :
        \text{$\mbf a$ is free and $\asdim_c(\mbf a\uhr \Delta) = \asdim(\Delta)$
            for every finitely generated $\Delta \le \Gamma$}\}
    \]
    is dense $G_\delta$,
    and hence comeager.
    
    In particular,
    if $\Gamma$ is locally of finite asymptotic dimension,
    then
    \[
        \{\mbf a \in \Act(\Gamma) :
        \text{$\mbf a$ is hyperfinite}\}
    \]
    is comeager.
\end{thm}
\begin{proof}
    For every finitely generated $\Delta \le \Gamma$,
    the set
    \[
        \{\mbf a \in \Act(\Gamma) :
        \text{$\mbf a$ is free and $\asdim_c(\mbf a\uhr \Delta) = \asdim(\Delta)$}\}
    \]
    is dense $G_\delta$ by \Cref{subgroup-asdim}.
    There are countably many finitely generated $\Delta \le \Gamma$,
    so the intersection over all of them is still dense $G_\delta$.

    Now suppose that $\Gamma$ is locally of finite asymptotic dimension.
    It suffices to show that every element $\mbf a$
    of this dense $G_\delta$ set is hyperfinite.
    Fix an increasing sequence $(\Delta_n)_n$ of finitely generated
    subgroups whose union is $\Gamma$.
    Then for every $n$,
    we have $\asdim_c(\mbf a \uhr \Delta_n) = \asdim(\Delta_n) < \infty$.
    Thus $\mbf a$ is hyperfinite by \cite[Theorem 7.3]{CJMST23}.
\end{proof}

\printbibliography

\end{document}